\documentclass{amsart}
\usepackage{amsmath,amssymb,hyperref,paralist}
\usepackage{amsrefs} 
\usepackage[all]{xy}
\usepackage[normalem]{ulem} 
\usepackage{mathrsfs}
\newtheorem{thm}{Theorem}

\newtheorem{cor}[thm]{Corollary}
\newtheorem{lemma}[thm]{Lemma}

\theoremstyle{definition}

\theoremstyle{remark}
\newtheorem{remark}[thm]{Remark}
\newtheorem{example}[thm]{Example}
\def\mathcs{C^{*}}
\newcommand{\cs}{\ensuremath{\mathcs}}

\DeclareMathSymbol{\rtimes}{\mathbin}{AMSb}{"6F}
\newcommand{\ib}{im\-prim\-i\-tiv\-ity bi\-mod\-u\-le}

\newcommand{\sme}{\,\mathord{\mathop{\text{--}}\nolimits_{\relax}}\,}

\DeclareMathOperator{\Ind}{Ind}

\DeclareMathOperator{\Prim}{Prim}
\DeclareMathOperator*{\supp}{supp}
\def\set#1{\{\,#1\,\}}
\newcommand\sset[1]{\{#1\}}
\let\tensor=\otimes
\def\restr#1{|_{{#1}}}
\makeatletter
\def\labelenumi{\textnormal{(\@alph\c@enumi)}}
\def\theenumi{\@alph \c@enumi}
\def\labelenumii{\textnormal{(\@roman\c@enumii)}}
\def\theenumii{\@roman \c@enumii}
\newcount\charno
\def\alphapart#1{\charno=96
\advance\charno by#1\char\charno}

\makeatother

\def\<{\langle}
\def\>{\rangle}
\let\ipscriptstyle=\scriptscriptstyle
\def\lipsqueeze{{\mskip -3.0mu}}
\def\ripsqueeze{{\mskip -3.0mu}}
\def\ipcomma{\nobreak\mathrel{,}\nobreak}
\newbox\ipstrutbox
\setbox\ipstrutbox=\hbox{\vrule height8.5pt
width 0pt}
\def\ipstrut{\copy\ipstrutbox}
\def\lip#1<#2,#3>{\mathopen{\relax_{\ipstrut\ipscriptstyle{
#1}}\lipsqueeze
\langle} #2\ipcomma #3 \rangle}
\def\blip#1<#2,#3>{\mathopen{\relax_{\ipstrut
\ipscriptstyle{ #1}}\lipsqueeze\bigl\langle} #2\ipcomma #3 \bigr\rangle}
\def\rip#1<#2,#3>{\langle #2\ipcomma #3
\rangle_{\ripsqueeze\ipstrut\ipscriptstyle{#1}}}
\def\brip#1<#2,#3>{\bigl\langle #2\ipcomma #3
\bigr\rangle_{\ripsqueeze\ipstrut\ipscriptstyle{#1}}}
\def\angsqueeze{\mskip -6mu}
\def\smangsqueeze{\mskip -3.7mu}
\def\trip#1<#2,#3>{\langle\smangsqueeze\langle #2\ipcomma #3
\rangle\smangsqueeze\rangle_{\ripsqueeze\ipstrut\ipscriptstyle{#1}}}
\def\btrip#1<#2,#3>{\bigl\langle\angsqueeze\bigl\langle #2\ipcomma
#3
\bigr\rangle
\angsqueeze\bigr\rangle_{\ripsqueeze\ipstrut\ipscriptstyle{#1}}}
\def\tlip#1<#2,#3>{\mathopen{\relax_{\ipstrut\ipscriptstyle{
#1}}\lipsqueeze \langle\smangsqueeze\langle} #2\ipcomma #3
\rangle\smangsqueeze\rangle}
\def\btlip#1<#2,#3>{\mathopen{\relax_{\ipstrut\ipscriptstyle{
#1}}\lipsqueeze
\bigl\langle\angsqueeze\bigl\langle} #2\ipcomma #3
\bigr\rangle\angsqueeze\bigr\rangle}

\def\ip(#1|#2){(#1\mid #2)}
\def\bip(#1|#2){\bigl(#1 \mid #2\bigr)}
\def\Bip(#1|#2){\Bigl( #1 \bigm| #2 \Bigr)}
\expandafter\ifx\csname BibSpec\endcsname\relax\else
\BibSpec{collection.article}{%
    +{}  {\PrintAuthors}                {author}
    +{,} { \textit}                     {title}
    +{.} { }                            {part}
    +{:} { \textit}                     {subtitle}
    +{,} { \PrintContributions}         {contribution}
    +{,} { \PrintConference}            {conference}
    +{}  {\PrintBook}                   {book}
    +{,} { }                            {booktitle}
    +{,} { }                            {series}
    +{,} { \voltext}                    {volume}
    +{,} { }                            {publisher}
    +{,} { }                            {organization}
    +{,} { }                            {address}
    +{,} { \PrintDateB}                 {date}
    +{,} { pp.~}                        {pages}
    +{,} { }                            {status}
    +{,} { \PrintDOI}                   {doi}
    +{,} { available at \eprint}        {eprint}
    +{}  { \parenthesize}               {language}
    +{}  { \PrintTranslation}           {translation}
    +{;} { \PrintReprint}               {reprint}
    +{.} { }                            {note}
    +{.} {}                             {transition}
}
\BibSpec{article}{%
    +{}  {\PrintAuthors}                {author}
    +{,} { \textit}                     {title}
    +{.} { }                            {part}
    +{:} { \textit}                     {subtitle}
    +{,} { \PrintContributions}         {contribution}
    +{.} { \PrintPartials}              {partial}
    +{,} { }                            {journal}
    +{}  { \textbf}                     {volume}
    +{}  { \PrintDatePV}                {date}
    +{,} { \eprintpages}                {pages}
    +{,} { }                            {status}
    +{,} { \PrintDOI}                   {doi}
    +{,} { available at \eprint}        {eprint}
    +{}  { \parenthesize}               {language}
    +{}  { \PrintTranslation}           {translation}
    +{;} { \PrintReprint}               {reprint}
    +{.} { }                            {note}
    +{.} {}                             {transition}
}
\BibSpec{book}{%
    +{}  {\PrintPrimary}                {transition}
    +{,} { \textit}                     {title}
    +{.} { }                            {part}
    +{:} { \textit}                     {subtitle}
    +{,} { \PrintEdition}               {edition}
    +{}  { \PrintEditorsB}              {editor}
    +{,} { \PrintTranslatorsC}          {translator}
    +{,} { \PrintContributions}         {contribution}
    +{,} { }                            {series}
    +{,} { \voltext}                    {volume}
    +{,} { }                            {publisher}
    +{,} { }                            {organization}
    +{,} { }                            {address}
    +{,} { pp.~}                        {pages}
    +{,} { \PrintDateB}                 {date}
    +{,} { }                            {status}
    +{}  { \parenthesize}               {language}
    +{}  { \PrintTranslation}           {translation}
    +{;} { \PrintReprint}               {reprint}
    +{.} { }                            {note}
    +{.} {}                             {transition}
}
\fi
\newcommand\go{G^{(0)}} 
\def\g[#1,#2]{{}_{G}[#1,#2]} 
\def\h[#1,#2]{[#1,#2]_{H}}
\newcommand\cc{C_{c}}

\newcommand\HH{\mathscr{H}}

\def\ipp(#1|#2){\ip({#1}|{#2})_{\pi}}

\newcommand\bundlefont\mathscr
\newcommand\B{\bundlefont{B}}
\newcommand\E{\bundlefont{E}}

\def\sa_#1(#2,#3){\Gamma_{#1}(#2;#3)}

\newcommand\sacgb{\sa_{c}(G,\B)}

\newcommand\half{\frac12}
\newcommand\dint{\int^{\oplus}}
\newcommand\atensor{\odot}
\renewcommand\H{\mathcal{H}}
\newcommand\mpi{M_{\pi}}
\newcommand\gugo{G\backslash \go}
\newcommand\gshr{G*\HH_{r}}
\DeclareMathOperator{\End}{End}

\usepackage{color}

\newcount\hours
\newcount\minutes       
\def\timeofday{
\hours=\time
\minutes=\hours
\divide\hours by60
\multiply\hours by60
\advance\minutes by-\hours
\divide\hours by60
\ifnum\hours>9\else0\fi\the\hours:\ifnum\minutes>9\else
0\fi\the\minutes}
\def\predate{\date{\color{red}\bfseries \the\day\ \ifcase\month\or
  January\or February\or March\or April\or May\or June\or July\or
        August\or September\or October\or November\or
           December\fi\ \the\year\ --- \timeofday -- v010}}

\begin{document}

\author{Aidan Sims}
\address{School of Mathematics and Applied Statistics\\
University of Wollongong \\
NSW 2522\\
Australia}

\email{asims@uow.edu.au}

\author{Dana P. Williams}
\address{Department of Mathematics \\ Dartmouth College \\ Hanover, NH
03755-3551}

\email{dana.williams@Dartmouth.edu}



\date{22 December 2011}

\title{Amenability for Fell bundles over groupoids}

\begin{abstract}
  We establish conditions under which the universal and reduced norms
  coincide for a Fell bundle over a groupoid. Specifically, we prove
  that the full and reduced $C^*$-algebras of any Fell bundle over a
  measurewise amenable groupoid coincide, and also that for a groupoid
  $G$ whose orbit space is $T_0$, the full and reduced algebras of a
  Fell bundle over $G$ coincide if the full and reduced algebras of
  the restriction of the bundle to each isotropy group coincide.
\end{abstract}

\maketitle
\tableofcontents

\section*{Introduction}\label{sec:introduction}

If $G$ is an amenable group, then the reduced crossed product and full
crossed product for any action of $G$ on a $C^*$-algebra
coincide. This result was proved for discrete groups by Zeller-Meyer
in \cite{zel:jmpa68} and in general by Takai in
\cite{tak:jfa75}. Since the \cs-algebra of a Fell bundle over a
groupoid $G$ is a very general sort of crossed product by $G$, it is
reasonable to expect the universal norm and reduced norm to coincide
on $\sacgb$ when $G$ is suitably amenable.  Immediately the situation
is complicated because amenability for groupoids is not as clear cut
as it is for groups.  There are three reasonable notions of
amenability for a second countable locally compact Hausdorff groupoid:
(topological) amenability, measurewise amenability and, for lack of a
better term, ``metric amenability'' by which we simply mean that the
reduced norm and universal norm on $\cc(G)$ coincide.

Amenability implies measurewise
amenability which in turn implies metric amenability. While there are
situations where the converses hold, it is
unknown if they hold in general. Our main result here,
Theorem~\ref{thm-main-amen}, is that if $G$ is measurewise amenable as
defined in \cite{anaren:amenable00}, then the reduced norm and
universal norm on $\sacgb$ coincide for any Fell bundle
  $\B$ over $G$. This result subsumes the usual result for group
dynamical systems and the result for groupoid dynamical systems; for a
discussion of this, see \cite{simwil:xx11}*{Examples 10~and 11}. The
result for groupoid systems is also asserted in
\cite{anaren:amenable00}*{Proposition~6.1.10} where they cite
\cite{ren:jot91}*{Theorem~3.6}. Since it is usually
hard to determine if a groupoid found in the
wild is amenable in any given one the three flavors mentioned above,
we also prove in Theorem~\ref{thm-smooth-amen} that groupoids which
act nicely on their unit spaces in the sense that that $\gugo$ is
$T_{0}$ and whose stability groups are all amenable are themselves
measurewise amenable. This result may be known
to experts, but seems worth advertising.
We also show that if $G$
is a groupoid whose orbit space is $T_0$ and if $\B$ is a Fell bundle
over $G$ such that the full and reduced $C^*$-algebras of the
restriction of $\B$ to each isotropy group in $G$ coincide, then the
full and reduced $C^*$-algebras of the whole bundle coincide. This is
a formally stronger result than the combination of
Theorem~\ref{thm-smooth-amen} and Theorem~\ref{thm-main-amen}: there
are many examples of Fell bundles over non-amenable groups whose full
and reduced $C^*$-algebras coincide (see, for example,
\cite{exe:jfram97}).

We start with very short sections on Fell bundles and amenable
groupoids to clarify our definitions and
point to the relevant literature.  In
Section~\ref{sec:disint-theor-revis} we point out a simple
strengthening of the disintegration theorem for Fell bundles (from
\cite{muhwil:dm08}) which is needed here.  For readability, the
details are shifted to Appendix~\ref{sec:borel--functors}.  In
Section~\ref{sec:fell-bundle-cs} we prove our main
theorem. In Section~\ref{sec:induced amenability} we show
that groupoids with $T_0$ orbit space and amenable stability groups
are measurewise amenable. In Section~\ref{sec:induced amenability} we
prove  that bundles over
groupoids with $T_0$ orbit space whose restrictions to isotropy groups
are metrically amenable are themselves metrically amenable.

Since we appeal to the disintegration theorem for Fell bundles, we
require separability for our results.  In particular, all the
groupoids and spaces that appear will be assumed to be second
countable, locally compact and Hausdorff.  Except when it is clearly
not the case, for example $B(\H)$ and other multiplier algebras, all
the algebras and Banach spaces that appear are separable. We
  also assume that our Fell bundles are always
saturated. The underlying Banach bundles are only required
to be upper semicontinuous.

\section{Fell bundles}
\label{sec:fell-bundles}

We will refer to \cite{muhwil:dm08}*{\S1} for details of the
definition of a Fell bundle $p : \B\to G$ over a groupoid as well as
of the construction of the associated \cs-algebra $\cs(G,\B)$. (The
examples in \cite{muhwil:dm08}*{\S2} would be very helpful
supplementary reading.) Roughly speaking, a Fell bundle $p : \B \to G$
is an upper-semicontinuous Banach bundle endowed with a partial
multiplication compatible with $p$ such that the fibres $A(u)$ over
units $u$ are $C^*$-algebras and such that each fibre $B(x)$ is an
$A(r(x))$--$A(s(x))$-\ib\ with respect to the inner products and
actions induced by the multiplication on $\B$.  In particular, when
$x$ and $y$ are composable, multiplication in $\B$ implements
isomorphisms $B(x) \otimes_{A(s(x))} B(y) \cong B(xy)$. The space
$\sa_{c}(G,\B)$ of continuous sections of $\B$ then carries a natural
convolution and involution. The $C^*$-algebra $\cs(G,\B)$ is the
completion of $\sa_{c}(G,\B)$ with respect to the universal norm for
representations which are continuous with respect to the
inductive-limit topology on $G$.

Regarding our notation: as above, we use a roman letter, $B(x)$, for
the fibre over $x$ together with its Banach space structure, but we
will use both $A(u)$ and $B(u)$ for the fibre over a unit $u$ so as to
distinguish its dual roles. The Fell bundle axioms imply that
$A:=\sa_{0}(\go,\B)$ is a \cs-algebra which is called the \cs-algebra
of $\B$ over $\go$; in particular it is a $C_{0}(\go)$-algebra. So for
$u \in \go$ we write $A(u)$ for the fibre over $u$ when we are
thinking of it as a \cs-algebra, and we write $B(u)$ when we are
thinking of it instead as an $A(u)\sme A(u)$-\ib. We assume that our
Fell bundles are separable, so in addition to $G$ being second
countable, we assume that the Banach space $\sa_{0}(G,\B)$ is
separable.  By axiom, our Fell bundles are saturated in that
$B(x)B(y)=B(xy)$, where  $B(x)B(y)$ denotes
  $\overline{\operatorname{span}}\set{b_x b_y : \text{$b_x \in B(x)$, $b_y \in
  B(y)$}}$. If $F$ is a locally closed subset\footnote{Recall that a
  subset of a locally compact Hausdorff space is locally compact if
  and only if it is locally closed and is locally closed if and only
  if it is the intersection of a closed set and an open set
  \cite{wil:crossed}*{Lemmas 1.25~and 1.26}.}  of $\go$, then we abuse
notation slightly and write $\sa_{c}(F,\B)$
in place of $\sa_{c}(F,\B\restr F)$ (as we have already done for
$A=\sa_{0}(\go,\B)$ above).  If we let $G(F) := G\restr F=\set{x\in
  G:\text{\(s(x)\in F\) and \(r(x)\in F\)}}$ be the restriction of $G$
to $F$, then $G(F)$ is a locally compact groupoid with Haar system
$\sset{\lambda^{u}}_{u\in F}$.  As above, we write $\cs(G(F),\B)$ in
place of
$\cs(G(F),\B\restr{G(F)})$.

Recall the definition of the reduced norm
on $\sacgb$ from \cite{simwil:xx11}. If $\pi$ is a representation of
$A=\sa_{0}(\go,\B)$, then using \cite{wil:crossed}*{Example~F.25} and
the discussion preceding \cite{muhwil:nyjm08}*{Definition~7.9}, we can
assume that there is a Borel Hilbert bumdle $\go * \HH$, a finite
Radon measure $\mu$ on $\go$ and representations $\pi_u$ of
$A$ on $\H(u)$, factoring through $A_u$, such that
\begin{equation*}
  \pi=\dint_{\go}\pi_{u}\,d\mu(u).
\end{equation*}
For $u \in \go$, we frequently regard the $\pi_{u}$ as
representations of $A(u)$. Even if
$\pi$ is nondegenerate, we can only assume that
  $\mu$-almost all of the $\pi_u$ are nondegenerate. Indeed, we could
have $\pi_u = 0$ for a null set of $u$. The formula $(\Ind\pi)(f)(g\tensor h)=(f*g)\tensor h$
  for $f,g \in \sacgb$ and $h \in L^2(\go * \HH, \mu)$ determines a
representation $\Ind\pi$ of $\sacgb$  on the
completion of $\sacgb\atensor L^{2}(\go*\HH,\mu)$ with respect to the
inner product
\begin{align}
  \ip(f\tensor h|g\tensor k)&= \bip(\pi(g^{*}*f)h|k)\notag \\
  &=\int_{\go}\int_{G}\bip(\pi_{u}\bigl(g(x^{-1})^{*}f(x^{-1})\bigr)
  h(u)|{k(u)}) \,d\lambda^{u}(x)\,d\mu(u)\notag \\
  &= \int_{\go}\int_{G}
  \bip(\pi_{u}\bigl(g(x)^{*}f(x)\bigr)h(u)|{k(u)}) \,d\lambda_{u}(x)
  \,d\mu(u).\label{eq:IndPi-ip}
\end{align}

The reduced norm on $\sacgb$ is given by
\begin{equation*}
  \|f\|_{r}:=\sup\set{\|(\Ind\pi)(f)\|:\text{$\pi$ is a representation
      of $A$}}.
\end{equation*}
Since $\ker (\Ind\pi)$ depends only on $\ker \pi$ (by, for example,
\cite{rw:morita}*{Corollary~2.73}) this definition of
$\|\cdot\|_{r}$ agrees with
 other definitions in the literature---for
example Exel's in \cite{exe:jfram97} and Moutuou and Tu's in
\cite{moutu:xx11}. So
  $\cs_{r}(G,\B)$ is the quotient of $\cs(G,\B)$ by the
kernel $I_{\cs_{r}(G,\B)}$ of $\Ind\pi$
for any faithful representation $\pi$ of the \cs-algebra
$A=\sa_{0}(\go,\B)$ of $\B$ over~$\go$.

\section{Amenable groupoids}
\label{sec:amenable-groupoids}

Let $G$ be a second-countable locally compact groupoid with Haar
system $\lambda$.  Renault \cite{ren:groupoid}*{p.~92} originally
defined $G$ to be topologically amenable, or just amenable, if there
is a net $\sset{f_{i}}\subset \cc(G)$ such
that
\begin{compactenum}
\item the functions $u\mapsto f_i * f^*_i(u)$ are uniformly bounded on
  $C_{0}(\go)$, and
\item $f_i * f^*_i \to 1$ uniformly on compact subsets of $G$.
\end{compactenum}
Later, in the extensive treatment by Anatharaman-Delaroche and
Renault, an \emph{a priori} different definition was given:
\cite{anaren:amenable00}*{Definition~2.2.8}; however,
\cite{anaren:amenable00}*{Proposition~2.2.13(iv)} and its proof show
that the two notions of amenability are equivalent.  It is not hard to
see, using standard criteria such as
\cite{wil:crossed}*{Proposition~A.17}, that a group is amenable as a
groupoid if and only if it is amenable as a group.

Let $\mu$ be a quasi-invariant measure on $\go$, and let $\nu :=
\mu\circ\lambda$ be the induced measure on $G$ (that is, $\nu(\cdot) =
\int_{\go} \lambda^u(\cdot)\,d\mu(u)$). In
\cite{anaren:amenable00}*{Definition~3.2.8}, $\mu$ is called amenable
if there exists a suitably invariant mean on $L^{\infty}(G,\nu)$. The
pair $(G,\lambda)$ is measurewise amenable if every quasi-invariant
measure $\mu$ is amenable
\cite{anaren:amenable00}*{Definition~3.2.8}. Since $L^{\infty}(G,\nu)$
depends only on the equivalence class of $\nu$, if $\mu'$ is
equivalent to $\mu$ and $\mu$ is amenable, then so is $\mu'$. Since
\cite{anaren:amenable00} considers only $\sigma$-finite measures, to
demonstrate that $(G,\lambda)$ is measurewise amenable, it suffices to
show that every $\emph{finite}$ quasi-invariant measure $\mu$ is
amenable.

It follows from \cite{anaren:amenable00}*{Theorem~2.2.17} and
\cite{anaren:amenable00}*{Theoerem~3.2.16} that amenability and
measurewise amenability, respectively, are preserved under groupoid
equivalence.  Theorem~17 of \cite{simwil:jot11} implies that
metric amenability is preserved as well.  In particular, none of the
three flavors of amenability of $G$ depend on the choice of
Haar system $\lambda$.

In this note, we will use the characterization of amenability of
$(G,\lambda,\mu)$ given in
\cite{anaren:amenable00}*{Proposition~3.2.14(v)}. If $G$ is amenable
then it is measurewise amenable by
\cite{anaren:amenable00}*{Proposition~3.3.5}. If $G$ is measurewise
amenable then it is metrically amenable by
\cite{anaren:amenable00}*{Proposition~6.1.8}.

\section{The disintegation theorem revisited}
\label{sec:disint-theor-revis}

Our main tool here is the disintegration theorem from \cite{muhwil:dm08}. Fix a nondegenerate
representation $L$ of $\cs(G, \B)$. Then \cite{muhwil:dm08}*{Theorem~4.13} implies that there are a
quasi-invariant measure $\mu$ on $\go$, a Borel Hilbert bundle $\go * \HH$, and a Borel $*$-functor
$b \mapsto \bigl(r(b),\pi(b),s(b)\bigr)$ (see \cite{muhwil:dm08}*{Definition~4.5}) from $\B$ into
$\operatorname{End}(\go*\HH)$ such that $L$ is equivalent to the integrated form of the associated
strict representation $(\mu, \go*\HH, \pi)$ of $\B$. For $h,k\in L^{2}(\go*\HH,\mu)$ and $f\in
\sacgb$, we then have
\begin{equation*}
  \bip(L(f)h|k)=\int_{G} \Bip(\pi(f(x))h(s(x))|{k(r(x)})
  \Delta(x)^{-\half}\,d \nu(x).
\end{equation*}

Regrettably, the authors of \cite{muhwil:dm08} neglected to point out
that the Borel $*$-functor associated to $L$ constructed in
\cite{muhwil:dm08}*{Theorem~4.13} is \emph{nondegenerate} in the sense
that for all $x\in G$,
\begin{equation}
  \label{eq:2}
  \overline{\pi(B(x))\H(s(x))}=\overline{\operatorname{span}}\set{\pi(b)v:
    \text{$b\in B(x)$ and $v\in \H(s(x))$}} =\H(r(x)).
\end{equation}

We outline why this is true in Appendix~\ref{sec:borel--functors}, and
at the same time, we tidy up some details of the proof of the
disintegration theorem itself.

\section{Fell bundles over amenable
  groupoids} \label{sec:fell-bundle-cs}

Our first main theorem says that every Fell bundle over a measurewise
amenable groupoid is metrically amenable.

\begin{thm}
  \label{thm-main-amen}
  Let $G$ be a second-countable locally compact Hausdorff groupoid
  with Haar system $\sset{\lambda^{u}}_{u\in\go}$.  Suppose that
  $p:\B\to G$ is a separable Fell bundle over $G$. If $G$ is
  measurewise amenable, then the reduced norm on $\sacgb$ is equal to
  the universal norm, so $\cs_{r}(G,\B)=\cs(G,\B)$.
\end{thm}

Our proof follows the lines of Renault's proof of the corresponding
result for groupoid $C^*$-algebras, suitably modified for the bundle
context. Before getting into the proof, we need to do a little
set-up. We will continue with the following notation for the remainder
of the section.

Fix $a\in I_{\cs_{r}(G,\B)}$ and let $L$ be a nondegenerate
representation of $\cs(G,\B)$. As in
Section~\ref{sec:disint-theor-revis}, we may assume that $L$ is the
integrated form of a strict representation
$\bigl(\mu,\go*\HH,\pi\bigr)$ of $\B$ which is nondegenerate in the
sense that~\eqref{eq:2} holds for all $x$. Fix a unit vector $h$ in
$L^{2}(\go*\HH,\mu)$\label{pg:main-thm-setup}, and let $\omega_h$ be
the associated vector state.  To prove Theorem~\ref{thm-main-amen}, it
suffices to see that $\omega_{h}(a)=0$.

Let $\gshr$ be the pullback of $\go*\HH$ over the range map. We may
describe it as follows. Let $(h_{j})^\infty_{j=1}$ be a special
orthorgonal fundamental sequence for $\go*\HH$ as in
\cite{wil:crossed}*{Proposition~F.6}. For each $j$, let $\tilde
h_{j}(x)=h_{j}(r(x)) \in \H(r(x))$.  Then $\gshr$ is isomorphic to the
Borel Hilbert bundle built from $\coprod_{x\in G}\H(r(x))$ with
fundamental sequence $(\tilde h_{j})^\infty_{j=1}$.

Let $\nu = \mu\circ\lambda$ be the measure on $G$ induced by $\mu$,
and recall that $\nu^{-1}$ denotes the measure $\nu^{-1}(f) = \int_G
f(x^{-1}) d\nu(x)$. Since $\mu$ is quasi-invariant, $\nu$ and
$\nu^{-1}$ are equivalent measures.  By passing to an equivalent
measure, we may assume that the Radon-Nikodym derivative
$\Delta=d\nu/d\nu^{-1}$ is multiplicative from $G$ to $(0,\infty)$ ---
there is a nice proof of this in
\cite{muh:cbms}*{Theorem~3.15}.\footnote{The proof in \cite{muh:cbms}
  unfortunatly remains unpublished, but it is based on Hahn's
  \cite{hah:tams78}*{Corollary~3.14}.}

\begin{lemma}
  \label{lem-key-unitary}
  Define $U : \sacgb\atensor L^{2}(\go*\HH,\mu) \to L^{2}(G * \HH_r,
  \nu^{-1})$ by $U(f\tensor h)(x) =
  \pi\bigl(f(x)\bigr)h\bigl(s(x)\bigr)$.  Then $U$ is isometric and
  extends to a unitary, also denoted by $U$, from $\H_{\Ind\pi_{\mu}}$
  onto $L^{2}(G*\HH_{r},\nu^{-1})$.  Furthermore, $U$ intertwines the
  regular representation $\Ind\pi_{\mu}$ with the representation
  $\mpi$ of $\cs(G,\B)$ on $L^{2}(G*\HH,\nu^{-1})$ given on $f\in
  \sacgb$ by
  \begin{equation}
    \label{eq:5}
    \bip(\mpi(f)\xi|\eta)=\int_{G}\int_{G}
    \Bip(\pi(f(xy))\xi(y^{-1})|{\eta(x)}) \,d\lambda^{s(x)}(y)\,d\nu^{-1}(x).
  \end{equation}

\end{lemma}
\begin{proof}
  That $\pi$ is a Borel $*$-functor, $f$ is a continuous section and
  \begin{align*}
    \bip({U(f\tensor h)(x)}|{\tilde h_{j}(x)})&=
    \bip(\pi(f(x))h(s(x))|{h_{j}(r(x)}) \\
    &=\sum^\infty_{k=1}\bip(h(s(x))|{h_{k}(s(x))})
    \bip(\pi(f(x))h_{k}(s(x))|{h_{j}(r(x))}),
  \end{align*}
  imply that $x\mapsto \bip(U(f\tensor h)(x)|{\tilde h_{j}(x)})$ is
  Borel.  Thus $U(g\tensor h) \in B(G*\HH_{r})$.  The representation
  $\pi_{\mu}$ comes from a Borel $*$-functor defined on all of $\B$,
  so the formula~\eqref{eq:IndPi-ip} for the inner product on
  $\H_{\Ind\pi_{\mu}}$ becomes
  \begin{align*}
    \bip(f\tensor h|g\tensor k) &= \int_{\go}\int_{G}
    \bip(\pi\bigl(g(x)^{*}f(x)\bigr) h(u)|{k(u)})
    \,d\lambda_{u}(x)\,d\mu(u) \\
    &=\int_{G}\bip(\pi(f(x))h(s(x))|{\pi(g(x))k(s(x))})\,d\nu^{-1}(x) \\
    &=\int_{G}\bip(U(f\tensor h)(x)|{U(g\tensor k)(x)})\,
    d\nu^{-1}(x).
  \end{align*}
  Hence $U(f\tensor h)\in L^{2}(G*\HH,\nu^{-1})$ and $U$ is an
  isometry.  Since $\pi$ is nondegenerate, an argument like that of
  \cite{wil:crossed}*{Lemma~F.17}, shows that the range of $U$ is
  dense. Hence, $U$ is a unitary as claimed.

  For the last assertion, recall that $\Ind\pi_{\mu}(f)$ acts by
  convolution.  Thus
  \begin{align*}
    \bip(\mpi(f)\xi|\eta)&=\int_{G}\bip(\mpi(f)\xi(x)|\eta{(x)})
    \,d\nu^{-1}(x) \\
    &= \int_{G}\int_{G}\bip(\pi(f(y))\xi(y^{-1}x)|{\eta(x)})
    \,d\lambda^{r(x)}(y) \,d\nu^{-1}(x) \\
    &=\int_{G}\int_{G} \bip(\pi(f(xy))\xi(y^{-1})|{\eta(x)})
    \,d\lambda^{s(x)}(y)\,d\nu^{-1}(x).\qedhere
  \end{align*}
\end{proof}

To prove Theorem~\ref{thm-main-amen}, we invoke measurewise
amenability in the form of
\cite{anaren:amenable00}*{Proposition~3.2.14(v)}. So we fix a sequence
$(f_{n})^\infty_{n=1}$ of Borel functions on $G$ such that
\begin{compactenum}
\item $u\mapsto \int_{G}|f_{n}(x)|^{2}\,d\lambda^{u}(x)$ is bounded on
  $\go$,
\item $f^*_n * f_n(u) \le 1$ for all $u\in\go$ and
\item $f^*_n * f_n \to 1$ in the weak-$*$ topology on
  $L^{\infty}(G,\nu)$.
\end{compactenum}
To keep notation compact, we denote $f^*_n * f_n$ by $e_n$, so that
for $y \in G$,
\begin{equation*}
  e_{n}(y) = \int_{G}\overline{f_{n}(x^{-1})} f_{n}(x^{-1}y) \,d\lambda^{r(y)}(x),
\end{equation*}

\begin{proof}[Proof of Theorem~\ref{thm-main-amen}]
  Recall the notation fixed at the beginning of the section: in
  particular, $L$ is the integrated form of a nondegenerate strict
  representation of $C^*(G,\B)$ on a Hilbert bundle $\go * G$, $h$ is
  a unit vector in $L^{2}(\go*\HH,\mu)$ and $\omega_h$ is the
  associated vector state. We claim that $|\omega_h(g)| \le
  \|(\Ind\pi_{\mu})(g)\|$ for all $g \in \sacgb$.

  Fix $g \in \sacgb$. Then
  \begin{equation*}
    \omega_{h}(g)=\bip(L(g)h|h)=\int_{G}\bip(\pi(g(y))h(s(y))|{h(r(y))})
    \Delta(y)^{-\half}\,d\nu(y).
  \end{equation*}
  Define a sequence $(\alpha_n)^\infty_{n=1}$ of complex numbers by
  \begin{align}
    \alpha_n &:= \int_{G}e_{n}(y)\bip(\pi(g(y))h(s(y))|{h(r(y))})
    \Delta(y)^{-\half}\,d\nu(y) \nonumber \\
    &= \int_{\go}\int_{G}\int_{G}
    \overline{f_{n}(x^{-1})}f_{n}(x^{-1}y)
    \bip(\pi(g(y))h(s(y))|{h(r(y))})
    \Delta(y)^{-\half} \label{eq:alpha_n expr}\\
    &\hskip3in\,d\lambda^{u}(x) \,d\lambda^{u}(y)\,d\mu(u).\nonumber
  \end{align}
  (It is tempting to write $\omega_{h}(e_ng)$ for $\alpha_n$, but the
  $e_n$ are assumed only to be Borel, so the pointwise products
  $e_{n}g$ may not belong to $\sacgb$.) By assumption on the $e_{n}$,
  the $\alpha_n$ converge to $\omega_{h}(g)$. So it suffices to show
  that
  \begin{equation}\label{eq:suffices}
    |\alpha_n| \le \|(\Ind\pi_{\mu})(g)\|\quad\text{ for all $n \in \mathbb{N}$.}
  \end{equation}

  Fix $n \in \mathbb{N}$. Define $h_n : G \to \HH$ by
  \[
  h_{n}(x)=\Delta(x)^{\half}f_{n}(x^{-1})h\bigl(r(x)\bigr).
  \]
  Then for each $j$, the function
  \begin{equation*}
    x\mapsto \bip(h_{n}(x)|{\tilde h_{j}(x)})
    =\Delta(x)^{\half} f_{n}(x^{-1}) \bip(h(r(x))|{h_{j}(r(x))})
  \end{equation*}
  is Borel, so $h_{n}\in L^{2}(G*\HH,\nu^{-1})$. Starting
  from~\eqref{eq:alpha_n expr}, we apply Fubini's theorem, substitute
  $xy$ for $x$, and then use first that $\nu = \Delta\nu^{-1}$ and
  then that $\Delta$ is multiplicative to calculate:
  \begin{align*}
    \alpha_n &= \int_{G}\int_{G} \overline{f_{n}(x^{-1})} f_{n}(y)
    \bip (\pi(g(xy)) h(s(y))|{h(r(x))}) \Delta(xy)^{-\half}
    \,d\lambda^{s(x)}(y) \,d\nu(x) \\
    &= \int_{G}\int_{G} \overline{f_{n}(x^{-1})}f_{n}(y) \bip(
    \pi(g(xy)) h(s(y)) | {h(r(x))}) \\
    &\hskip12em \Delta(xy)^{-\half}\Delta(x)
    \,d\lambda^{s(x)}(y) \,d\nu^{-1}(x) \\
    &=\int_{G}\int_{G}
    \bip(\pi(g(xy))h_{n}(y^{-1})|{h_{n}(x)})\,d\lambda^{s(x)}(y)
    \,d\nu^{-1}(x)\\
    &=\bip(\mpi(g)h_{n}|h_{n}).
  \end{align*}

  We have
  \[
  \|h_{n}\|^{2} = \int_{G}\|h_{n}(x)\|^{2}\,d\nu^{-1}(x)
  =\int_{G}|f_{n}(x^{-1})|^{2} \|h(r(x))\|^{2} \Delta(x)
  \,d\nu^{-1}(x),
  \]
  and since $\nu=\Delta\nu^{-1}$ and $e_n(u) \le 1$ for all $u$, it
  follows that
  \[
  \|h_{n}\|^2 =\int_{\go}e_{n}(u)\|h(u)\|^{2}\,d\mu(u) \le \|h\|^2 =
  1.
  \]

  Hence the Cauchy-Schwarz inequality gives~\eqref{eq:suffices}. Thus
  $|\omega_h(g)| \le \|(\Ind\pi_\mu)(g)\|$ for all $g \in \sacgb$ as
  claimed. Since $\sacgb$ is dense in $\cs(G,\B)$, it follows that
  \begin{equation*}
    |\omega_{h}(a)|\le \|(\Ind\pi_{\mu})(a)\|\quad\text{for all $a\in \cs(G,\B)$.}
  \end{equation*}
  In particular, if $a\in I_{\cs_{r}(G,\B)}$, then $(\Ind\pi_{\mu})(a)
  = 0$, and hence $\omega_{h}(a)=0$ as required.
\end{proof}

\begin{example}\label{eg:DKPS}
  Recall from \cite{dkps:iumj11} that given a row-finite $k$-graph
  $\Lambda$ with no sources, a $\Lambda$-system of
  $C^*$-correspondences consists of an assignment $v \mapsto A_v$ of
  $C^*$-algebras to vertices and an assignment $\lambda \mapsto
  X_\lambda$ of an $A_{r(\lambda)}$--$A_{s(\lambda)}$ correspondence
  to each path $\lambda$, together with isomorphisms $\chi_{\mu,\nu} :
  X_\mu \otimes_{A_{s(\mu)}} X_\nu \to X_{\mu\nu}$ for each composable
  pair $\mu,\nu \in \Lambda$, all subject to an appropriate
  associativity condition on the $\chi_{\mu,\nu}$ (see
  \cite{dkps:iumj11}*{Definitions 3.1.1~and~3.1.2} for
  details). Suppose that $X$ is such a system, and suppose that each
  $X_\lambda$ is nondegenerate as a left $A_{r(\lambda)}$-module, and
  full as a right Hilbert $A_{s(\lambda)}$-module, and that the left
  action of $A_{r(\lambda)}$ is by compact operators.

  By \cite{dkps:iumj11}*{Theorem~4.3.1}, the construction of Sections
  4.1~and~4.2 of the same paper associates to $X$ a saturated Fell
  bundle $E_X$ over the $k$-graph groupoid $G_\Lambda$ of
  \cite{kumpas:nyjm00}. Moreover, \cite{dkps:iumj11}*{Theorem~4.3.6}
  says that the $C^*$-algebra $C^*(A,X,\chi)$ of the $\Lambda$-system
  is isomorphic to the reduced $C^*$-algebra $\cs_{r}(G_\Lambda, E_X)$
  of the Fell bundle.

  Theorem~5.5 of \cite{kumpas:nyjm00} says that $G_\Lambda$ is
  amenable, and hence also measurewise amenable. Hence our
  Theorem~\ref{thm-main-amen} implies that $\cs_{r}(G_\Lambda, E_X) =
  \cs(G_\Lambda, E_X)$; in particular $\cs(A,X,\chi) \cong
  \cs(G_\Lambda, E_X)$.
\end{example}

Since $1$-graphs are precisely the path-categories $E^*$ of countable
directed graphs $E$, and since an $E^*$-system of correspondences can
be constructed from any assignment of $C^*$-algebras $A_v$ to vertices
$v$, and $A_{r(e)}$--$A_{s(e)}$ $C^*$-correspondences $X_e$ to edges
$e$ (see \cite{dkps:iumj11}*{Remark~3.1.5}), Example~\ref{eg:DKPS}
provides a substantial library of examples of our result

\section{Measurewise amenable groupoids}
\label{sec:meas-amen-group}

Our initial motivation for proving Theorem~\ref{thm-main-amen} was to
show that if $G$ has $T_0$ orbit space and amenable stability groups
then the full and reduced \cs-algebras of any Fell bundle over $G$
coincide: roughly, since $\cs(G,\B)$ is a $C_0(\gugo)$-algebra,
representations will factor through restrictions to orbit groupoids
$G([u])$, each of which is amenable because is is equivalent to the
amenable stability group $G(u) := \{x \in G : r(x) = u = s(x)\}$
(see section~\ref{sec:induced amenability} for
  details). However, the following argument shows that
 the result follows directly from
Theorem~\ref{thm-main-amen}. We thank Jean Renault for pointing us in
the direction of \cite{anaren:amenable00}*{Proposition~5.3.4}.

\begin{thm}
  \label{thm-smooth-amen}
  Suppose that $G$ is a second countable locally compact Hausdorff
  groupoid with Haar system $\sset{\lambda^{u}}_{u\in\go}$.  Suppose
  that the orbit space $\gugo$ is $T_{0}$ and that each stability
  group $G(u)$ is amenable.  Then $G$ is measurewise amenable.
\end{thm}

Our proof requires some straightforward observations as well as some
nontrivial results from \cite{anaren:amenable00}.

\begin{lemma}
  \label{lem-reduce}
  Suppose that $\mu$ is a quasi-invariant finite measure on $\go$ and
  that $F\subset \go$ is a locally compact $G$-invariant subset such
  that $\mu(\go\setminus F)=0$. Then $(G,\lambda,\mu)$ is amenable if
  and only if $(G(F),\lambda\restr F,\mu\restr F)$ is amenable.
\end{lemma}
\begin{proof}
  Recall that $(G,\lambda,\mu)$ is amenable if there is an invariant
  mean on $L^{\infty}(G,\nu)$ where $\nu=\mu\circ\lambda$.  Since
  $\mu\restr F\circ \lambda\restr F=\nu\restr {G(F)}$, we have
  $L^{\infty}(G,\nu) \cong L^{\infty}(G(F), \mu\restr F \circ
  \lambda\restr F)$.  In particular, an invariant mean on
  $L^{\infty}(G)$ gives an invariant mean on $L^{\infty}(G(F))$
  and vice versa.
\end{proof}

\begin{lemma}
  \label{lem-ramsay}
  Suppose that $\gugo$ is $T_{0}$.  Then, as a Borel space, $\gugo$ is
  countably separated and each orbit $[u]$ is locally closed in $G$
  and hence locally compact.
\end{lemma}
\begin{proof}
  Since subsets of a locally compact Hausdorff space are locally
  compact if and only if they are locally closed (see
  \cite{wil:crossed}*{Lemma~1.26}), the lemma is an immediate
  consequence of the Mackey-Glimm-Ramsay dichotomy
  \cite{ram:jfa90}*{Theorem~2.1}.
\end{proof}

\begin{proof}[Proof of Theorem~\ref{thm-smooth-amen}]
  Suppose that $\mu$ is a finite quasi-invariant measure on $\go$.  It
  suffices to show that $(G,\lambda,\mu)$ is amenable.  Let
  $p:G\to\gugo$ be the orbit map, and let $\underline\mu$ be the
  forward image $\underline\mu(f) = \mu(f \circ p)$ of $\mu$ under
  $p$.  By Lemma~\ref{lem-ramsay}, $\gugo$ is countably
  separated as a Borel space. Hence we can disintegrate $\mu$ --- as,
  for example, in \cite{wil:crossed}*{Theorem~I.5} --- so that for
  each orbit $[u]$ there is a probability measure $\rho_{[u]}$ on
  $\go$ supported on $[u]$ such that
  \begin{equation*}
    \mu=\int_{\gugo}\rho_{[u]}\,d\underline\mu([u]).
  \end{equation*}
  It follows from \cite{anaren:amenable00}*{Proposition~5.3.4} that
  $\rho_{[u]}$ is quasi-invariant for almost all $[u]$ and that
  $(G,\lambda,\mu)$ is amenable if each $(G,\lambda,\rho_{[u]})$
  is. Since $\rho_{[u]}(\go\setminus [u])=0$, Lemma~\ref{lem-reduce}
  implies that it is enough to see that each
  $(G([u]),\lambda\restr{[u]}, \mu\restr{[u]})$ is amenable.
  Since $[u]$ is locally compact, $G([u])$ is a locally compact
  transitive groupoid equivalent to $G(u)$, which is assumed to be
  amenable. Hence \cite{anaren:amenable00}*{Theorem~2.2.13} implies
  that $G([u])$ is amenable, and therefore also measurewise
  amenable by \cite{anaren:amenable00}*{Proposition~3.3.5}.
\end{proof}

\section{Fibrewise-amenable Fell bundles}\label{sec:induced
  amenability}

In the preceding section, we showed that if $\gugo$ is $T_0$ and each
stability group is amenable, then $G$ is measurewise amenable. In
particular, if $p : \B \to G$ is a bundle over such a groupoid, then
its full and reduced algebras coincide. In this section, we show that
it suffices that $\gugo$ is $T_0$ and that for each $u \in \go$, the
full and reduced algebras of the restriction of $\B$ to the isotropy
group $G(u)$ coincide. To see that this is a strictly stronger
theorem, and also that the hypothesis is genuinely checkable, we refer
the reader to the results, for example, of \cite{exe:jfram97}.

\begin{thm}\label{thm:induced metric amenability}
  Let $G$ be a second-countable locally compact Hausdorff groupoid
  with Haar system $\sset{\lambda^{u}}_{u\in\go}$, and let $p : \B \to
  G$ be a separable Fell bundle over $G$. Suppose that the orbit space
  $\gugo$ is $T_{0}$ and that for each unit $u$, the full and reduced
  cross-sectional algebras $\cs(G(u),\B)$ and $\cs_{r}(G(u),\B)$
  coincide. Then the full and reduced norms on $\sacgb$ are equal and
  hence $\cs_{r}(G,\B) = \cs(G,\B)$.
\end{thm}

To prove the theorem, we first use the equivalence theorem of
\cite{simwil:xx11} to see that the full and reduced $C^*$-algebras of
a Fell bundle over transitive groupoid coincide whenever the full and
reduced algebras of its restriction to any isotropy group
coincide.

\begin{lemma}\label{lem:fibre equivalence}
  Let $G$ be a second-countable locally compact Hausdorff groupoid
  with Haar system $\sset{\lambda^{u}}_{u\in\go}$, and let $p : \B \to
  G$ be a separable Fell bundle over $G$. Suppose that $G$ is
  transitive. Then the following are equivalent.
    \begin{enumerate}
    \item For some unit $u$, the full and reduced cross-section
      algebras $\cs(G(u),\B)$ and $\cs_{r}(G(u),\B)$ coincide.
    \item For every unit $u$, the full and reduced cross-section
      algebras $\cs(G(u),\B)$ and $\cs_{r}(G(u),\B)$ coincide.
    \item The full and reduced norms on $\sacgb$ are equal and hence
      $\cs_{r}(G,\B) = \cs(G,\B)$.
    \end{enumerate}
\end{lemma}
\begin{proof}
  Fix $u\in\go$.  Then $G_{u}:=s^{-1}(u)$ is a
  $(G,G(u))$-equivalence, and as in \cite{ionwil:ms11}*{Theorem~1},
  $\E:=p^{-1}(G_{u})$ implements an equivalence between $\B$ and
  $p^{-1}(G(u))$.  Consequently, \cite{simwil:xx11}*{Theorem~14}
  implies that the natural surjection of $\cs(G,\B)$ onto
  $\cs_{r}(G,\B)$ is an isomorphism if and only if the kernel $I_{r}$
  of the natural map of $\cs(G(u),\B)$ onto $\cs_{r}(G(u),\B)$ is
  trivial. Since $u\in\go$ was arbitrary, the result follows.
\end{proof}

To finish off our proof of Theorem~\ref{thm:induced metric
  amenability}, we need the following special case of
\cite{ionwil:hjm11}*{Theorem~3.7}.
As above, let $p : \B \to
  G$ be a separable Fell bundle over $G$ with associated \cs-algebra
  $A=\sa_{0}(\go,\B)$.  Recall from
  \cite{ionwil:hjm11}*{Proposition~2.2} that $G$ acts on $\Prim A$
  which we identify with $\set{(u,P):\text{$u\in\go$ and $P\in\Prim
      A(u)$}}$.
Let $U$ be an open
  $G$-invariant subset of $\go$ with complement $F$. Then $\set{(u, P)
    \in\Prim A:
  u \in F}$ is a closed invariant subset of $\Prim(A)$,
  and corresponds to the $G$-invariant ideal $\set{a \in A : \text{$a(u) =
  0$ for all $u \in F$}}$ of $A$. By
  \cite{ionwil:hjm11}*{Proposition~3.3}, the corresponding bundle
  $\B_I$ is the one with fibres
  \[
  B_I(x) =
  \begin{cases}
    B(x) &\text{ if $x \in G(U)$}\\
    \{0\} &\text{ if $u \in G(F)$,}
  \end{cases}
  \]
  so we can identify it with the bundle $\B|_{G(U)}$ over
  $G(U)$. Moreover, $\B^I$ is the complementary bundle
  \[
  B^I(x) =
  \begin{cases}
    \{0\} &\text{ if $u \in G(U)$}\\
    B(x) &\text{ if $u \in G(F)$,}
  \end{cases}
  \]
  which we may identify with the bundle $\B|_{G(F)}$ over
  $G(F)$. Thus, as a special case of Theorem~3.7 of
  \cite{ionwil:hjm11}, we obtain the following result.
  \begin{lemma}
    \label{lem-special-ionwil}
    Let $G$ be a second-countable locally compact Hausdorff groupoid
  with Haar system $\sset{\lambda^{u}}_{u\in\go}$, and let $p : \B \to
  G$ be a separable Fell bundle over $G$.   Suppose that $U$ is a
  $G$-invariant open subset of $\go$ with complement $F$.  There is
  a short exact sequence of \cs-algebras
  \begin{equation*}
    \xymatrix{0\ar[r]&\cs(G(U),\B)\ar[r]^-{\iota}&\cs(G,\B)\ar[r]^-{q}&
\cs(G(F),\B)\ar[r]&0,}
  \end{equation*}
where $\iota$ is induced by inclusion and $q$ by restriction on sections.
  \end{lemma}

As an application of Lemma~\ref{lem-special-ionwil}, recall\footnote{A
  proof can be constructed along the lines of
\cite{muhwil:dm08}*{Proposition~4.2}.} that there is a nondegenerate
map $M:C_{0}(\go)\to
M(\cs(G,\B))$ given on sections by
\begin{equation*}
  M(\phi)f(x)=\phi\bigl(r(x)\bigr)f(x).
\end{equation*}
Suppose that the orbit space $\gugo$ is Hausdorff. Then we may
identify $C_{0}(\gugo)$ with the subalgebra of $C_b(\go)$ consisting
of functions which are constant on orbits and vanish at infinity on
the orbit space. We extend $M$ to $C_{b}(\go)$ and restrict to
$C_{0}(\gugo)$ to obtain a nondegenerate map of $C_{0}(\gugo)$ into
the center of $M(\cs(G,\B))$, making $\cs(G,\B)$ into a
$C_{0}(\gugo)$-algebra. As usual, if $u\in\go$, we let $[u]$ be the
corresponding orbit in $\gugo$.
\begin{cor}
  \label{cor-hausdorff-go}
  Let $G$ be a second-countable locally compact Hausdorff groupoid
  with Haar system $\sset{\lambda^{u}}_{u\in\go}$, and let $p : \B \to
  G$ be a separable Fell bundle over~$G$.  If $\gugo$ is Hausdorff,
  then $\cs(G,\B)$ is a $C_{0}(\gugo)$-algebra with fibres
  $\cs(G,\B)([u])\cong \cs(G([u]),\B)$.
\end{cor}
\begin{proof}
  Recall that $\cs(G,\B)([u])$ is the quotient of $\cs(G,\B)$ by the
  ideal $J_{[u]}=\overline{\operatorname{span}}\set{\phi\cdot
    a:\text{$\phi\in C_{0}(\gugo)$, $\phi([u])=0$ and
      $a\in\cs(G,\B)$}}$.   Using Lemma~\ref{lem-special-ionwil}, we
  can identify $J_{[u]}$ with  $\cs(G(U),\B)$, where
  $U=\go\setminus [u]$, and $\cs(G,\B)/J_{[u]}$ with $\cs(G([u]),\B)$
  as claimed.
\end{proof}

\begin{proof}[Proof of Theorem~\ref{thm:induced metric amenability}]
  Fix in irreducible representation $\pi$ of $\cs(G;\B)$ and an
  element $f \in \sacgb$. It suffices to show that $\|\pi(f)\|
  \le \|f\|_{\cs_{r}(G;\B)}$.

  By \cite{ram:jfa90}*{Theorem~2.1}, the orbit space $\gugo$ is
  locally Hausdorff and every orbit $[u]$ is locally closed in
  $\go$. Since $\gugo$ is second countable,
  \cite{wil:crossed}*{Lemma~6.3} implies that there is a countable
  ordinal $\gamma$ and a nested open cover $\set{U_n : 0\le n\le
    \gamma }$ of $\gugo$ such that $U_{0}=\emptyset$,
  $U_{\gamma}=\gugo$ and $U_{n+1}\setminus U_{n}$ is Hausdorff (and
  dense) in $(\gugo)\setminus U_{n}$.  For $n \le \gamma$, let $V_n :=
  \set{u \in \go : [u] \in U_n}$. Then each $V_n$ is an open invariant
  subset of $\go$. Using Lemma~\ref{lem-special-ionwil}, we can
  identify $\cs(G(V_{n}),\B)$ with an ideal in $\cs(G,\B)$.  In fact,
  $\set{\cs(G(V_{n}),\B)}_{n\le\gamma}$ is a composition series of
  ideals in $\cs(G,\B)$. By \cite{wil:crossed}*{Lemma~8.13}, there
  exists $0<n \le \gamma$ such that $\pi$ lives on the subquotient
  $\cs(G(V_n), \B)/\cs(G(V_{n-1}), \B)$; that is, $\pi$
  is the canonical lift $\bar \rho$ of an irreducible representation
  $\rho$ of the ideal $\cs(G(V_{n}),\B)$ such that $\ker \rho\supset
  \cs(G(V_{n-1}),\B)$.  Lemma~\ref{lem-special-ionwil} implies
  that $\cs(G(V_n), \B)/\cs(G(V_{n-1}), \B) \cong \cs(G(V_{n}\setminus
  V_{n-1}),\B)$. By construction, $G(V_{n}\setminus V_{n-1})$ has Hausdorff orbit
  space $U_{n}\setminus U_{n-1}$.  Hence $\cs(G(V_{n}\setminus
  V_{n-1}),\B)$ is a $C_{0}(U_{n}\setminus U_{n-1})$-algebra and
  $\rho$ factors through a fibre $\cs(G(V_{n}\setminus
  V_{n-1}),\B)([u])\cong \cs(G([u]),\B)$ for some $u\in V_{n}$ by
  \cite{wil:crossed}*{Proposition~C.5}.  Since, by assumption,
  $\cs(G[u]),\B)= \cs_{r}(G([u]),\B)$, we have $\ker(\Ind \pi_{[u]})
  \subset \ker\rho$ where $\pi_{[u]}$ factors through a faithful
  representation of the quotient $A_{V_{n}}([u])$ of the \cs-algebra
  $A(V_{n})$ of $\cs(G(V_{n}),\B)$ corresponding to the closed set
  $[u]\subset V_{n}$.  (Note that $A(V_{n})$ is the ideal of $A$
  corresponding to $V_{n}\subset \go$.)  The kernel of $\pi=\bar\rho$ is
    then contained in the kernel of the canonical lift of
    $\Ind\pi_{[u]}$ to $\cs(G,\B)$.  It is not hard to check that the
  canonical lift of $\Ind\pi_{[u]}$ to $\cs(G,\B)$ is $\Ind
  \bar\pi_{[u]}$ where $\bar \pi_{[u]}$ is the canonical lift of
  $\pi_{[u]}$ to $A$. Hence $\|\pi(f)\|=\|\bar\rho(f)\|\le
  \|\Ind\bar\pi_{[u]}(f)\|
\le \|f\|_{\cs_{r}(G,\B)}$.
\end{proof}

\appendix

\section{Nondegenerate Borel \texorpdfstring{$*$}{*}-functors}
\label{sec:borel--functors}

Let $p:\B\to G$ be a Fell bundle over a second-countable locally
compact Hausdorff groupoid, and let $\go*\HH$ be a Borel Hilbert
bundle. A \emph{Borel $*$-functor} $\hat\pi$ from $\B$ to
$\End(\go*\HH)$ is a map
\[
\hat\pi : b \mapsto \bigl(r(b), \pi(b), s(b)\bigr)
\]
such that $\pi(b) \in B\bigl(\H(s(b)), \H(r(b))\bigr)$ for all $b$ and
such that $\pi$ respects adjoints and the partial linear and
multiplicative structure of $\B$ (see
\cite{muhwil:dm08}*{Definition~4.5}). Following
\cite{muhwil:dm08}*{\S4}, a strict representation of $\B$ is a triple
$(\mu,\go*\HH,\hat\pi)$ consisting of a quasi-invariant measure $\mu$
on $\go$, a Borel Hilbert bundle $\go*\HH$ and a Borel $*$-functor
$\hat\pi$. It is common practice to use $\hat \pi$ and $\pi$
interchangeably, and we will drop the caret henceforth. A strict
representation determines a bounded representation via integration
(see \cite{muhwil:dm08}*{Proposition~4.10}); indeed, a Borel
$*$-functor defined on $p^{-1}(G\restr F)$ for any $\mu$-conull set
$F\subset \go$ is sufficient. Nevertheless, it is convenient to have
$\pi$ defined everywhere.

The purpose of this section is to point out that the disintegration
theorem \cite{muhwil:dm08}*{Theorem~4.13} for Fell bundles can be
strengthened to assert that $\pi$ can be taken to be nondegenerate as
defined in \S\ref{sec:disint-theor-revis}. At the same time, we
correct an error in the construction of $\pi$ in \cite{muhwil:dm08}.

In the proof of \cite{muhwil:dm08}*{Theorem~4.13}, starting from a
pre-representation $L$ of $\B$ on a dense subspace $\H_0$ of a Hilbert
space $H$, the authors showed that for any orthonormal
basis $\{\zeta_i : i \in \mathbb{N}\}$ for
$\operatorname{span}\{L(f)\xi : f \in \sacgb, \xi \in \H_0\}$, setting
$\H'_{00} := \operatorname{span}\{\zeta_i : i \in \mathbb{N}\}$, there
is a saturated Borel $\mu$-conull set $F\subset \go$ and a Borel
Hilbert bundle $F*\HH$ whose fibres $\H(u)$ are Hilbert-space
completions of $\sa_{c}(G^{u},\B)\atensor\H'_{00}$ (see
\cite{muhwil:dm08}*{Lemma~5.18} and
\cite{muhwil:dm08}*{Lemma~5.20}). For $f \in \sacgb$ and $h \in
\H'_{00}$, the class of $f\tensor h$ in $\H(u)$ is denoted by
$f\tensor_{u}h$.  The space $\H(u)$ may be trivial for some
$u$. For each $z \in G\restr F$, $b\in B(z)$ and $f \in
  \sacgb$, let $\check\pi(b)f$ denote a section satisfying
\begin{equation*}
  \check\pi(b)f(x)=\Delta(z)^{\half} b f(z^{-1}x)\quad\text{for $x\in G^{r(b)}$.}
\end{equation*}
The Borel $*$-functor in the disintegration of $L$
constructed in \cite{muhwil:dm08} is defined by
\begin{equation*}
  \pi(b)(f\tensor_{s(b)}\zeta_{i})=\check\pi(b)f\tensor_{r(b)}\zeta_{i}.
\end{equation*}
Since $F$ is saturated, $G$ is the disjoint union of $G\restr F$ and
$G\restr {\go\setminus F}$.  Since the latter is $\nu$-null, we can
extend $\pi$ to all of $G$ by defining it as we please on
$p^{-1}(G\restr{\go\setminus F})$, and this will not affect the
integrated representation. To ensure that $\pi$ is still a genuine
Borel $*$-functor, one sets $\H(u) := \{0\}$ for each $u\notin F$ and
$\pi(b) := 0$ for $b\notin p^{-1}(G\restr F)$. (In \cite{muhwil:dm08},
the authors mistakenly let $(\go\setminus F)*\HH$ be a constant field
and let $\pi(b)$ be the identity operator, but such a $\pi$ is not a
$*$-functor since as it doesn't preserve the partial linear
structure.) We claim that $\pi$ is nondegenerate in the sense
that~\eqref{eq:2} holds for all $z \in G$. It holds trivially for $z
\not\in G\restr F$, so fix $z\in G\restr F$, and let $u := r(z)$. We
start with two observations.
\begin{enumerate}[(A)]
\item If $f_{i}\to f$ in the inductive limit topology on
  $\sa_{c}(G^{u},\B)$ then $f_{i}\tensor_{u}\zeta_{k}\to
  f\tensor_{u}\zeta_{k}$ in $\H(u)$. To see this, observe that
  equation~(5.19) of \cite{muhwil:dm08} is bounded by
  $K\|f\|_{\infty}\|g\|_{\infty}$ where $K$ is constant depending only
  on $\supp f$ and $\supp g$.
\item If $\set{e_{i}}$ is an approximate identity in $A(u)$, and, for
  each $i$, $e_{i}g$ represents any section in $\sa_{c}(G,\B)$ such
  that $(e_{i}g)(x)=e_{i}g(x)$ for $x\in G^{u}$, then $e_{i}g\to g$ in
  the inductive limit topology on $\sa_{c}(G^{u},\B)$. This follows
  from a compactness argument using that $A(u)$ acts nondegenerately
  on $B(x)$.
\end{enumerate}

By~(B), to establish~\eqref{eq:2} for $z$, it suffices to see that
each $e_{i}g\tensor_{r(z)}\zeta_{k}$ belongs to
$\overline{\pi(B(z)\H(s(z))}$. Fix $b_{1},\dots b_{n}\in B(z)$ such
that
\begin{equation*}
  \sum_{j}b_{j}b_{j}^{*}\sim e_{i}.
\end{equation*}
Then by~(A), we have
\begin{equation*}
  \sum_{j}\pi(b_{j})(\check\pi(b_{j}^{*})g\tensor \zeta_{k})\sim
  e_{i}g\tensor \zeta_{k},
\end{equation*}
and this suffices.

\begin{remark}
  \label{rem-notneeded}
  Just as $*$-functors are automatically bounded (see
  \cite{muhwil:dm08}*{Remark~4.6}), there is a sense in which the
  Borel $*$-functor appearing in any strict representation
  $(\mu,\go*\HH,\pi)$ is essentially nondegenerate.  We claim that
  \begin{equation}
    \label{eq:3}
    \overline{\pi\bigl(B(x)\bigr)\H\bigl((s(x)\bigr)}=
    \overline{\pi\bigl(A(r(x)\bigr)
      \H\bigl((r(x)\bigr)} \quad\text{for all $x\in G$.}
  \end{equation}
  The right-hand side of~\eqref{eq:3} is the essential space of the
  representation $\pi_{r(x)}$ of $A\bigl(r(x)\bigr)$ determined by
  $\pi$, so~\eqref{eq:2} holds whenever $\pi_{r(x)}$ is nondegenerate.
  So if the representation $\pi_{\mu}$ of $A=\sa_{0}(\go,\B)$
  determined by $\pi$ is nondegenerate, then $\pi_{u}$ is
  nondegenerate for $\mu$-almost all $u$, so~\eqref{eq:2} holds on a
  $\nu$-conull subset of $G$ (where, as usual, $\nu =
  \mu\circ\lambda$).

  To establish~\eqref{eq:3}, we use that $\B$ is saturated: one the
  one hand,
  \begin{equation*}
    \overline{\pi\bigl(B(x)\bigr)\H\bigl(s(x)\bigr)} =
    \overline{\pi\bigl(A(r(x)\bigr)\bigr) \pi\bigl(B(x)\bigr)
      \H\bigl(s(x)\bigr)} \subset
    \overline{\pi\bigl(A(r(x)\bigr)\bigr)\H\bigl(r(x)\bigr)},
  \end{equation*}
  while on the other hand,
  \begin{equation*}
    \overline{\pi\bigl(B(x)\bigr) \H\bigl(s(x)\bigr) } \supset
    \overline{\pi\bigl(B(x)
      \bigr)\pi\bigl(B(x^{*})\bigr)\H\bigl(r(x)\bigr)}
    =\overline{\pi\bigl(A\bigl(r(x) \bigr)\bigr)\H\bigl(r(x)\bigr)}.
  \end{equation*}
\end{remark}


\def\noopsort#1{}\def\cprime{$'$} \def\sp{^}

\end{document}